\documentclass[pdflatex,sn-mathphys-num]{sn-jnl}

\usepackage{graphicx}%
\usepackage{multirow}%
\usepackage{amsmath,amssymb,amsfonts}%
\usepackage{amsthm}%
\usepackage{mathrsfs}%
\usepackage[title]{appendix}%
\usepackage{xcolor}%
\usepackage{textcomp}%
\usepackage{manyfoot}%
\usepackage{booktabs}%
\usepackage{algorithm}%
\usepackage{algorithmicx}%
\usepackage{algpseudocode}%
\usepackage{listings}%

\usepackage{url}

\theoremstyle{thmstyleone}%
\theoremstyle{thmstyletwo}%

\theoremstyle{thmstylethree}%
 \newtheorem{thm}{Theorem}[section]
 
 \newtheorem{lem}[thm]{Lemma}
 
 \theoremstyle{definition}
 
 \theoremstyle{remark}
 \newtheorem{rem}[thm]{Remark}
 
 \numberwithin{equation}{section}

\begin{document}

\title[Optimal bounds for the ratio of differences of means]{Optimal bounds for the ratio of differences of quadratic, arithmetic, and harmonic means}


\author*[1]{\fnm{Yagub} N. \sur{Aliyev}}\email{yaliyev@ada.edu.az}

\author[2]{\fnm{Narmin} N. \sur{Aliyeva}}\email{nermin.aliyeva@idrak.edu.az}

\author[1]{\fnm{Zamina} E. \sur{Guliyeva}}\email{zeguliyeva@ada.edu.az}

\equalcont{These authors contributed equally to this work.}

\affil*[1]{\orgdiv{School of IT and Engineering}, \orgname{ADA University}, \orgaddress{\street{61 Ahmadbay Agha-Oglu Street}, \city{Baku }, \postcode{AZ1008}, \country{Azerbaijan}}}

\affil[2]{\orgdiv{Department of Differential Equations and Control Theory}, \orgname{Baku State University}, \orgaddress{\street{Academic Zahid Khalilov str. 23}, \city{Baku}, \postcode{AZ1148}, \country{Azerbaijan}}}


\abstract{We determine the optimal constants in inequalities comparing the differences of the quadratic, arithmetic, and harmonic means of $n$ nonnegative real numbers. Specifically, we prove that for $n\ge3$ the sharp double inequality
\[
\frac{1}{\sqrt{n}}\le \frac{A_n-H_n}{Q_n-H_n}\le \sqrt{\frac{n-1}{n}}
\]
holds true.
This extends earlier results by T. Mitev, which established the sharp bounds only for the cases $n=3,4,$ and $5$. Our approach is based on a variant of the classical optimization method of Cauchy and Maclaurin, in which a quadratic symmetric function is optimized under simultaneous constraints on the arithmetic and harmonic means.
}

\keywords{Arithmetic mean,
Harmonic mean,
Quadratic mean,
Mean inequalities,
Sharp inequalities,
Cauchy--Maclaurin method.}

\pacs[MSC Classification (AMS)]{Primary 26E60;
Secondary 26D15.}

\maketitle

\section{Introduction}\label{sec1}

Let $H_n$, $A_n$ and $Q_n$ be the harmonic, arithmetic and quadratic means, respectively, of nonnegative real numbers $x_1,\ldots,x_n$:

$$H_n=\frac{n}{\sum\limits_{i =
1}^n{\frac{1}{x_{i}}}},\ 
A_n=\frac{\sum\limits_{i =
1}^n{x_{i}}}{n},\ Q_n=\sqrt{\frac{\sum\limits_{i =
1}^n{x_{i}^2}}{n}}.
$$
If there is an $i$ such that $x_i=0$, then we assume that $H_n=0$ (see \cite{hardy}, p. 12). In \cite{mitev1} and \cite{mitev2}
the best constants for the inequalities $C_1\le \frac{A_n-H_n}{Q_n-H_n}\le C_2$ are studied for the cases $n=3,4,5$. In \cite{mitev1} and \cite{mitev2} it was shown that if $n=3$, then $C_1\le\frac{\sqrt{3}}{3}$ and $C_2\ge\frac{\sqrt{6}}{3}$. Similarly, if $n=4$, then $C_1\le\frac{1}{2}$ and $C_2\ge\frac{\sqrt{3}}{2}$. Finally, if $n=5$, then $C_1\le\frac{\sqrt{5}}{5}$ and $C_2\ge\frac{2\sqrt{5}}{5}$.
In the current paper, we generalize these results by considering all the cases $n\ge 3$ in a unified way. The following result is obtained.

 \begin{thm}
     If $n\ge 3$, then  the inequalities $C_1\le \frac{A_n-H_n}{Q_n-H_n}\le C_2$ hold true for all nonnegative real numbers $x_1,\ldots,x_n$ if and only if $C_1\le\frac{1}{\sqrt{n}}$ and $C_2\ge\sqrt{\frac{n-1}{n}}$. Equality in the left inequality holds if and only if $C_1=\frac{1}{\sqrt{n}}$, and all $x_i$ are equal to zero except one. Equality in the right inequality holds if and only if $C_2=\sqrt{\frac{n-1}{n}}$, and all $x_i$ are equal to each other except one which is zero.
 \end{thm}

The main contribution of this paper is a complete solution of the problem for arbitrary $n\ge3$. We determine the best possible constants and characterize all equality cases. For the proof we use a variant of classical optimization argument by Cauchy \cite{cau} (see also \cite{polya1}, p. 64) and Maclaurin \cite{Mc}. In the classical variant $x_1x_2$ is maximized while keeping $x_1+x_2$ as constant. In the current paper we optimize $x_1^2+x_2^2+x_3^2$ while keeping $x_1+x_2+x_3$ and $\frac{1}{x_1}+\frac{1}{x_2}+\frac{1}{x_3}$ as constant (see \cite{aliyev2025extremal, aliyev1, aliyev2}).

The inequalities of the above type have applications in geometry for proving inequalities involving volume, radii of circumscribed and inscribed spheres, and other elements of $n$-simplexes \cite{wu0}, and in linear algebra for proving inequalities involving permanent and the eigenvalues of matrices \cite{wen2}. The means $H_n$, $A_n$ and $Q_n$ frequently appear in probability, statistics, and their applications in natural and social sciences. Therefore, the obtained result above can be interpreted in many different contexts for various purposes.

\section{Lemmas}

We will need the following limit to continuously extend $\frac{A_n - H_n}{Q_n - H_n}$ to its discontinuity points $x_1=\cdots =x_n$.
\begin{lem}
    \textit{If $x>0$, then$$ \lim_{(x_1, x_2, \dots, x_n) \to (x, x, \dots, x)} \frac{A_n - H_n}{Q_n - H_n} = \frac{2}{3}. \eqno(2.1)$$}
\end{lem}
\begin{proof}
Let us first calculate the following limit
   \begin{align*}
& \lim_{(x_1, x_2, \dots, x_n) \to (x, x, \dots, x)} \frac{A_n-H_n}{Q_n-A_n}\\
& = \lim_{(x_1, x_2, \dots, x_n) \to (x, x, \dots, x)} \frac{(A_n-H_n)(Q_n+A_n)}{Q_n^2-A_n^2}.
   \end{align*}
   Without loss of generality, assume $x=1$ and $A_n=1$. Hence, $x_i\to 1$ and $x_i=1+\alpha_i$, where, $\alpha_i\to 0.$
   Using these, we obtain
   \begin{align*}
H_n &= \frac{n}{\sum_{i=1}^{n} \frac{1}{x_i}} = \frac{n}{\sum_{i=1}^{n} \frac{1}{1+\alpha_i}}= \frac{n}{\sum_{i=1}^{n}\left(1 - \alpha_i + \alpha_i^2 + o(\alpha_i^2)\right)}  \\
&= \frac{n}{n - \sum_{i=1}^{n} \alpha_i + \sum_{i=1}^{n} \alpha_i^2 + \sum_{i=1}^{n} o(\alpha_i^2)} \\
&= \frac{1}{1 - \dfrac{\sum_{i=1}^{n} (\alpha_i - \alpha_i^2 + o(\alpha_i^2))}{n}} \\
&= 1 + \frac{\sum_{i=1}^{n} (\alpha_i - \alpha_i^2 + o(\alpha_i^2))}{n} + \frac{\left(\sum_{i=1}^{n} (\alpha_i - \alpha_i^2 + o(\alpha_i^2))\right)^2}{n^2} \\
&\quad + o\left(\left(\frac{\sum_{i=1}^{n} (\alpha_i - \alpha_i^2 + o(\alpha_i^2))}{n}\right)^2\right) \\
&= 1 - \frac{\sum_{i=1}^{n} \alpha_i^2}{n} + o\left(\sum_{i=1}^{n} \alpha_i^2\right).
\end{align*}
The above formula for $H_n$ can also be adopted from one of the exercises presented in \cite{statistics} (p. 150, Ex. 6.13). One can directly check that
   \begin{align*}
       Q_n^2=A_n^2+\frac{\sum_{i=1}^{n}(x_i-A_n)^2}{n}=1+\frac{\sum_{i=1}^{n}\alpha_i^2}{n}.
   \end{align*}
The same result for $Q_n^2$ can be obtained using more general results from \cite{scheib}.
Then,
   \begin{align*}
         \lim_{(\alpha_1, \alpha_2, \dots, \alpha_n) \to (0, 0, \dots, 0)} \frac{\left[ A_n - 1 + \frac{1}{n}\sum_{i=1}^{n} \alpha_i^2 + o\left(\sum_{i=1}^{n} \alpha_i^2\right) \right] \cdot (Q_n + A_n)}{\frac{\sum_{i=1}^{n} \alpha_i^2}{n}} = \\[2ex]
= \lim_{(\alpha_1, \alpha_2, \dots, \alpha_n) \to (0, 0, \dots, 0)} \frac{\left[ \frac{\sum_{i=1}^{n} \alpha_i^2}{n} + o\left(\sum_{i=1}^{n} \alpha_i^2\right) \right]}{\frac{\sum_{i=1}^{n} \alpha_i^2}{n}} \cdot 2 = 2 \\[3ex]
   \end{align*}
Therefore,
   \begin{align*}
       \lim_{(x_1, x_2, \dots, x_n) \to (x, x, \dots, x)}\frac{1}{\frac{1}{\frac{A_n - H_n}{Q_n - H_n}} - 1} = 2,
   \end{align*}
from which (2.1) follows.
\end{proof}
The following lemma justifies the optimization argument with 3 variables.
\begin{lem}
    \textit{Let $x, y$, and $z$ be positive real numbers such that $x+y+z=\beta$ and $\frac{1}{x}+\frac{1}{y}+\frac{1}{z}=\alpha$, where $\alpha$ and $\beta$ are positive constants such that ${\alpha}{\beta}\ge9$. Suppose also that $x,y,z$ satisfy $x\le y\le z$ in such a way that that at least one of the inequalities is strict. Then the function $g(x,y,z)=x^2+y^2+z^2$ decreases when $y$ increases, and the maximum and minimum of the function $g(x,y,z)$ occurs when $x=y<z$ and $x<y=z$, respectively.}
\end{lem}

\begin{proof}
    $x,y,z$ satisfying $x+y+z=\beta$ and $\frac{1}{x}+\frac{1}{y}+\frac{1}{z}=\alpha$ can be parametrized as
    \begin{align*}
        x = \frac{\beta - t \pm \sqrt{(\beta - t)^2 - \frac{4t(\beta - t)}{\alpha t - 1}}}{2},
        y=t,
        z = \frac{\beta - t \pm \sqrt{(\beta - t)^2 - \frac{4t(\beta - t)}{\alpha t - 1}}}{2},
    \end{align*}
    where, $t\in [t_1,t_2]$, and $t_1\in \left(0,\frac{\beta}{3}\right),$ $t_2\in \left(\frac{\beta}{3},\beta\right)$ are the zeros of the quadratic function $\kappa(t)=\alpha t^2-(\beta \alpha -3)t+\beta$.
    Indeed, $\kappa(0)= \beta>0$, $\kappa(\beta)=\alpha\beta^2-(\alpha\beta -3)\beta+\beta=\alpha\beta^2-\alpha\beta^2+4\beta=4\beta>0$, $\kappa(\frac{\beta}{3})=\alpha\frac{\beta^2}{9}-(\alpha\beta-3)\frac{\beta}{3}+\beta=\frac{\alpha\beta^2}{9}-\frac{\alpha\beta^2}{3}+2\beta=-\frac{2\alpha\beta^2}{9}+2\beta<0$. Therefore, the roots $t_1$ and $t_2$ of $\kappa(t)$ are located as it is shown above. The part of the curve satisfying $x\le y\le z$ is parametrized as
$$
x = \frac{\beta - t - \sqrt{(\beta - t)^2 - \frac{4t(\beta - t)}{\alpha t - 1}}}{2},\ y=t,\ z = \frac{\beta - t + \sqrt{(\beta - t)^2 - \frac{4t(\beta - t)}{\alpha t - 1}}}{2},
$$
where $t\in [t^*_1,t^*_2]$, and $t^*_1\in \left(t_1,\frac{\beta}{3}\right)$, $t^*_2\in \left(\frac{\beta}{3},t_2\right)$ are the zeros of the quadratic function $\kappa^*(t)=2 \alpha t^2-(\beta \alpha+3)t+2\beta$ (see Figure 1). Note that, $\kappa^*(0)=2\beta>0$, $\kappa^*(\beta)=2\alpha\beta^2-(\alpha\beta+3)\beta+2\beta=\alpha\beta^2-2\beta=\beta(\alpha\beta-2)>0$, $\kappa^*(\frac{\beta}{3})=2\alpha\frac{\beta^2}{9}-(\alpha\beta+3)\frac{\beta}{3}+2\beta=-\frac{\alpha\beta^2}{9}+\beta=\frac{9-\alpha\beta}{9}<0$. Moreover, $\kappa^*-\kappa=\alpha t^2-6t+\beta>0$. Indeed, the discriminant is $\Delta=36-4\alpha\beta<0$. Therefore, the roots $t_1^*$, $t_2^*$ of $\kappa^*(t)$ are between the roots $t_1$, $t_2$ of $\kappa(t)$. Specifically, the boundary values $t=t^*_1$ and $t=t^*_2$ yield $x=y<z$ and $x<y=z$, respectively, while the strict inequality $x<y<z$ holds for all interior points $t\in(t^*_1, t^*_2)$.

According to Vieta's formulas, $x$ and $z$ serve as the roots of the quadratic equation $w^2-(\beta -t)w+\frac{t(\beta-t)}{\alpha t-1}=0$. Differentiating this relation implicitly with respect to $t$ yields the derivative $w'_t=\frac{\beta +\alpha t^2-2t-w(\alpha t-1)^2}{(\alpha t-1)^2(2w-\beta +t)}$.

\begin{figure}[htbp]
\centering
\begin{minipage}[c]{0.4\textwidth}
  \centering
  \includegraphics[width=\textwidth]{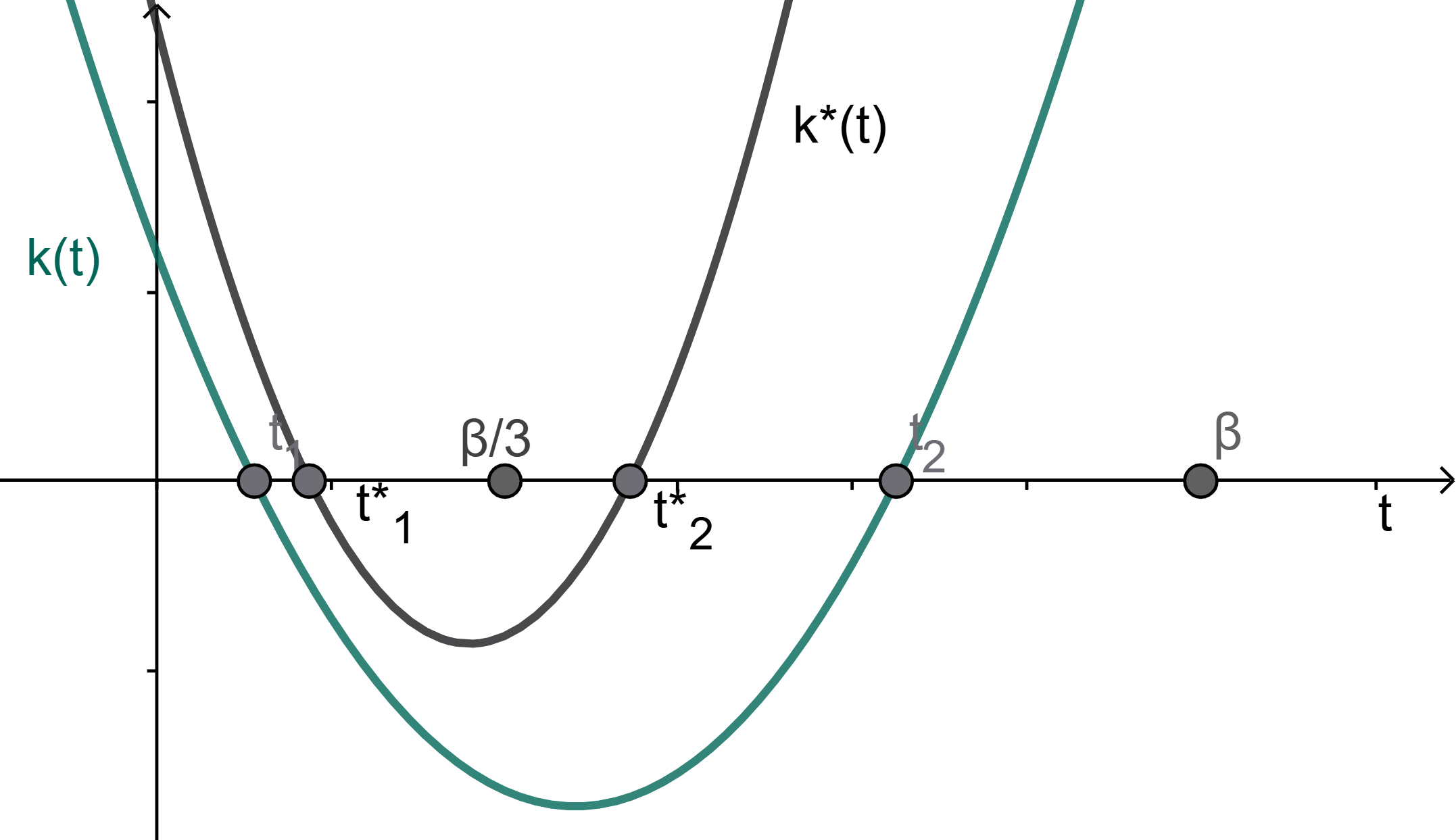}
\end{minipage}
\hspace{0.001cm}
\begin{minipage}[c]{0.55\textwidth}
  \centering
  \includegraphics[width=\textwidth]{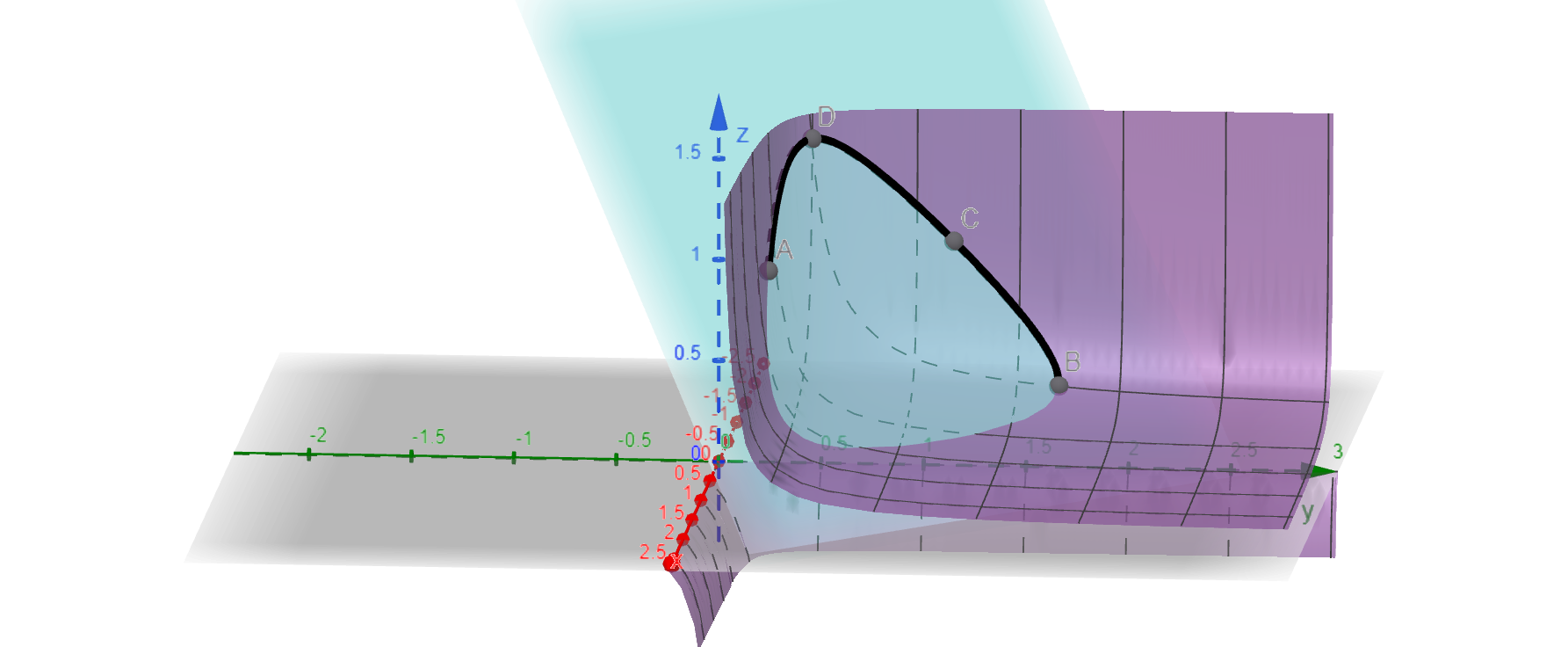}
\end{minipage}
\caption{Left: The graphs of $\kappa(t)$ (green), $\kappa^*(t)$ (black) and their zeros $t_1$, $t_2$, $t^*_1$, $t^*_2$, (black). Right: The graph of the parametric curve (black) representing the intersection of the plane (blue) $x+y+z=\beta$ and surface (purple) $\tfrac{1}{x}+\tfrac{1}{y}+\tfrac{1}{z}=\alpha$.}
\label{fig1}
\end{figure}

So, 
\begin{align*}
    x'_t=\frac{\beta +\alpha y^2-2y-x(\alpha y-1)^2}{(\alpha y-1)^2(2x-\beta +y)}  \text{ and }  z'_t=\frac{\beta +\alpha y^2-2y-z(\alpha y-1)^2}{(\alpha y-1)^2(2z-\beta +y)}.
\end{align*}
Or
\begin{align*}
    x'_t=\frac{x^2(z^2-y^2)}{y^2(x^2-z^2)}  \text{ and }  z'_t=\frac{z^2(x^2-y^2)}{y^2(x^2-z^2)}.
\end{align*}
Analyzing the signs of the derivatives, we observe that as $y$ increases, $x(t)$ decreases ($x'_t < 0$) and $z(t)$ also decreases ($z'_t < 0$). This monotonic behavior of the coordinates highlights how the components of the curve evolve with respect to the parameter $y$.
Then,
\begin{align*}
    g'_t=\frac{2x[\frac{\beta +\alpha y^2-2y}{(\alpha y-1)^2}-x]}{2x-\beta +y}+2y+\frac{2z[\frac{\beta +\alpha y^2-2y}{(\alpha y-1)^2}-z]}{2z-\beta +y}.
\end{align*}
After simplifications we obtain that
\begin{align*}
    g'_t=\frac{-2}{y^2(z+x)}(y-x)(z-y)(yx+yz+zx)<0.
\end{align*}
Consequently, as $y$ increases $g$ decreases.
\end{proof}
We will use the following variant of quotient rule for differentiation:
$$
\left(\frac{f_1}{f_2}\right)'=\left(\frac{f'_1}{f'_2}-\frac{f_1}{f_2}\right)\frac{f'_2}{f_2}.
$$
The following lemma known as the 'L’Hôpital-type Monotonicity Rule' will be essential. This rule is frequently employed in the investigation of optimization problems of the type considered here (see e.g. \cite{pin3}, \cite{ander1}, \cite{hardy}, p. 106, \cite{aliyev2025extremal}).
\begin{lem} Let $-\infty<a<+\infty$, and let $f_1,\ f_2:[a,b]\rightarrow \mathbb{R}$ be continuous on $[a,b]$ and differentiable on $(a,b)$. Let $f'_2(x)\ne0$ on $(a, b)$. Then, if $f'_1(x)/f'_2(x)$ is increasing (decreasing) on $(a , b)$, so are
$$[f_1(x) - f_1(a)] /[f_2(x) - f_2(a)] \text{ and } [f_1(x) - f_1(b)] /[f_2(x) - f_2(b)].$$
If $f'_1(x)/f'_2(x)$ is strictly monotone, then the monotonicity in the conclusion is also strict.
\end{lem}

\section{Proof of Theorem 1.1}
We will prove the indicated upper and lower bounds for the ratio $\frac{A_n-H_n}{Q_n-H_n}.$ Without loss of generality, assume $\sum_{i=1}^{n}x_i=1.$ 
Denote
$$
f(x_1,\ldots,x_n)=\begin{cases}

{\frac{2}{3}}, & \text{if } x_i=\frac{1}{n}\ (i=1,2,\ldots,n);\\
{\frac{Q_n-\frac{1}{n}}{H_n-\frac{1}{n}}}, & \text{if otherwise.}

\end{cases}
$$
By Eq. (2.1), function $f(x_1,\ldots,x_n)$ is continuous at point $\left(\frac{1}{n}, \frac{1}{n}, \ldots, \frac{1}{n}\right)$. Consequently, the function $f(x_1,\ldots,x_n)$ is both non-vanishing and continuous over the compact set

$$
\mathbb{D}=\{(x_{1},x_{2},\ldots ,x_{n})|\ x_{1},x_{2},\ldots ,x_{n}\ge 0;\sum\limits_{i =
1}^n{x_{i}}=1\},
$$
and
$\frac{A_n-H_n}{Q_n-H_n}=\frac{1}{1-f(x)}$. According to a fundamental theorem in mathematical analysis (see, e.g., \cite{rudin}, Theorem 4.16; \cite{zorich}, Sect. 7.2, p. 424), any continuous real-valued function defined on a compact $\mathbb{D}$ in space $\mathbb{R}^n$ achieves its maximum and minimum values on $\mathbb{D}$. Therefore, there exists a point $(x_1, \dots, x_n) \in \mathbb{D}$ at which $f$ attains its global maximum (or minimum). For this particular point, we may assume without loss of generality that $x_1 \leq x_2 \leq \dots\leq x_n$. We will consider 2 cases $x_1>0$ and $x_1=0$.

\noindent \textbf{The case $x_1>0$.} By Lemma 2.2 for any 3 coordinates $x_i$ at least 2 are equal. As a result the variables can be written as $$x_{1}=x_{2}=\ldots =x_{n-1}=x,\ x_{n}=1-(n-1)x,$$ where $0\le x\le\frac{1}{n-1}$ (see \cite{aliyev1}, \cite{aliyev2025extremal}, \cite{aliyev2}). Hence, it suffices to analyze the function $f(x)=\frac{q(x)-\frac{1}{n}}{h(x)-\frac{1}{n}}$ in the interval $0\le x\le\frac{1}{n-1}$, where 
\begin{align*}
    q(x)=\sqrt{\frac{(n-1)x^2+(1-(n-1)x)^2}{n}}, \\
    h(x)=\frac{n}{\frac{n-1}{x}+\frac{1}{1-(n-1)x}}.
\end{align*}
We calculate 
\begin{align*}
    q'(x)=\frac{(n-1)(nx-1)}{(n-1)x^2+(1-(n-1)x)^2}q(x), \\[1ex]
    h'(x)=\frac{(n-1)(1-(n-2)x)(1-nx)}{x(n-1-n(n-2)x)(1-(n-1)x)}h(x), \\
    q''(x)=\frac{q'(x)}{((n-1)x^2+(1-(n-1)x)^2)(nx-1)}, \\
    h''(x)=\frac{-2h'(x)}{(n-1-(n^2-2n)x)(1-nx)(1-(n-2)x)}.
\end{align*}

Therefore,
\begin{align*}
    \left(\frac{q'(x)}{h'(x)}\right)'=\frac{q'(x)}{h'(x)}\left(\frac{(n-1-n(n-2)x)(1-(n-2)x)}{2((n-1)x^2+(1-(n-1)x)^2)}-1\right)\frac{h''(x)}{h'(x)}
\end{align*}

Let us investigate how the signs of the parts of the product above change in accordance with the values of $x$. First note that
\begin{align*}
    q'(x)=\frac{(n-1)(nx-1)}{(n-1)x^2+(1-(n-1)x)^2} q(x).
\end{align*}
Since $n\ge 3$, $n-1>0$ always. Thus, the denominator is also always positive. Then, we only need to check the sign of $nx-1$.

$$
\begin{cases}

{nx-1<0}, & \text{if } x\in\left(0, \frac{1}{n} \right);\\
{nx-1>0}, & \text{if }  x\in\left(\frac{1}{n}, \frac{1}{n-1} \right).

\end{cases}
$$
Hence,
$$
\begin{cases}

{q'(x)<0}, & \text{if } x\in\left(0, \frac{1}{n} \right);\\
{q'(x)>0}, & \text{if }  x\in\left(\frac{1}{n}, \frac{1}{n-1} \right).

\end{cases}
$$
\begin{align*}
    h''(x)=\frac{-2h(x)}{x(n-1-(n-2)x)^2(1-(n-1)x)}.
\end{align*}
The denominator is always positive for $x\in\left(0, \frac{1}{n-1}\right).$ Therefore, $h''(x)<0$ for $x\in\left(0, \frac{1}{n-1}\right).$
We simplify the expression in parentheses:
\[
\frac{(n-1-n(n-2)x)(1-(n-2)x)}{2((n-1)x^2+(1-(n-1)x)^2)}-1
\]\[=\frac{((n^2-6n+6)x-n+3)(nx-1)}{2((n-1)x^2+(1-(n-1)x)^2}.
\]
The denominator is always positive and the factor $(n^2-6n+6)x-n+3$ in the numerator is always negative for $n\ge3$. Hence, the whole expression is positive when $x\in\left(0, \frac{1}{n}\right)$ and negative when $x\in\left(\frac{1}{n}, \frac{1}{n-1}\right).$
Thus, $ \left(\frac{q'(x)}{h'(x)}\right)'>0$ for $x\in\left(0, \frac{1}{n-1}\right)$ and $ \frac{q'(x)}{h'(x)}$ increases in this interval.
 By Lemma 2.3, $f(x)=\frac{q(x)-\frac{1}{n}}{h(x)-\frac{1}{n}}$ also increases in the interval $0< x<\frac{1}{n-1}$. Also, $\lim_{x\to 0^+}f(x)=1-\sqrt{n}$ and $\lim_{x\to \frac{1}{n-1}^-}f(x)=1-\sqrt{\frac{n}{n-1}}$.
Noting $\frac{A_n-H_n}{Q_n-H_n}=\frac{1}{1-f(x)}$ and using the results for $f(x)$ we can prove the bounds for $\frac{A_n-H_n}{Q_n-H_n}$ in Theorem 1.1.

\noindent \textbf{The case $x_1=0$.} We want to prove that again
$\frac{1}{\sqrt{n}}\le \frac{A_n-H_n}{Q_n-H_n}\le\frac{\sqrt{n-1}}{\sqrt{n}}$, where $0=x_1\le x_2\le \cdots\le x_n$.
If $x_1=0$, then $H_n=0$, and the inequality is written as
\begin{align*}
    \frac{1}{\sqrt{n}}\le \frac{A_n}{Q_n}\le \frac{\sqrt{n-1}}{\sqrt{n}}, \\[1ex]
    \frac{1}{\sqrt{n}}\le \frac{\frac{x_2+\cdots +x_n}{n}}{\sqrt{\frac{x_2^2+\cdots +x_n^2}{n}}}\le \frac{\sqrt{n-1}}{\sqrt{n}}.
\end{align*}
If we look at the left side of the inequality, we get $\sqrt{x_2^2\cdots +x_n^2}\le x_2+\cdots +x_n$, which is a true inequality and can be proved by raising both sides to the second power. The right side, on the other hand, can be proved by using Cauchy-Schwarz inequality (see e.g. \cite{mitrin}, p. 30). That is, 
\begin{align*}
    x_2+\cdots+ x_n\le \sqrt{x_2^2+\cdots +x_n^2}\cdot \sqrt{n-1}, \\[1ex]
    x_2\cdot 1+\cdots +x_n\cdot 1\le \sqrt{x_2^2+\cdots +x_n^2}\cdot \sqrt{\underbrace{1^2 + 1^2 + \dots + 1^2}_{{n-1 \text{ times}}}}.
\end{align*}
We conclude that the bounds are true for the case $x_1=0$ as well.

\begin{rem}
It would be interesting to study the more general question about
the best constants for the inequalities $C_1\le \frac{A_n-H_n}{P_\alpha-H_n}\le C_2$ for all possible values of $\alpha$. Here $P_\alpha$ is the power mean of non-negative real numbers $x_1,\ldots,x_n$:
$$
P_\alpha(x_1,\ldots,x_n)=\left(\frac{\sum\limits_{i =
1}^n{x_{i}^\alpha}}{n}\right)^\frac{1}{\alpha}.
$$
\end{rem}

\section{Conclusion} In the paper the inequalities $\frac{1}{\sqrt{n}}\le \frac{A_n-H_n}{Q_n-H_n}\le \sqrt{\frac{n-1}{n}}$ $(n\ge3)$ are proved.  It is shown that the constants $\frac{1}{\sqrt{n}},\ \sqrt{\frac{n-1}{n}}$ are the best possible. The obtained results generalize the results of T. Mitev for small values of $n$.

\section*{Acknowledgments}
The authors thank for the license of Maple 2025
software provided as a part of Maple Ambassador Program.

\section{Declarations}
\textbf{Ethical Approval.}
Not applicable.
 \newline \textbf{Competing interests.}
None.
  \newline \textbf{Authors' contributions.} 
Not applicable.
  \newline \textbf{Funding.}
This work was completed with the support of ADA University Faculty Research and Development Fund.
  \newline \textbf{Availability of data and materials.}
Not applicable



\end{document}